\documentclass[a4j]{article}
    \usepackage{amsthm,geometry}
    \usepackage{amsmath, amssymb,amsfonts, mathtools,ascmac,subcaption,color}
\usepackage{graphicx}
    \usepackage{caption,subcaption,comment}
    \usepackage{bbm,bm}
    \numberwithin{equation}{section}
    \theoremstyle{definition}
    \newtheorem{dfn}{Definition}[section]
    
    \newtheorem{lem}[dfn]{Lemma}
    \newtheorem{thm}[dfn]{Theorem}
    \newtheorem{cor}[dfn]{Corollary}
    \newtheorem{ass}{Assumption}
    \newtheorem{rem}{Remark}
    
    \newcommand{\E}{\mathsf{E}}
    \newcommand{\Prob}{\mathsf{P}}
    \newcommand{\Pois}{\mathrm{Pois}}
    \newcommand{\Bin}{\mathrm{B}}
    
    \newcommand{\bias}{\mathrm{bias}}
    \newcommand{\biasPois}{\mathrm{bias}^\mathrm{Pois}}
    
    \newcommand{\dx}{\mathrm{d}x}
    \newcommand{\dy}{\mathrm{d}y}

\title{Asymptotic bias of the plug-in Shannon entropy estimator \\ under a regularly varying occupancy model}
\author{Takato Hashino\thanks{Joint Graduate School of Mathematics for Innovation, Kyushu University, Fukuoka, Japan.} \and Koji Tsukuda\thanks{Faculty of Mathematics, Kyushu University, Fukuoka, Japan.}}
\date{}
\begin{document}
\maketitle

\begin{abstract}
Estimating the Shannon entropy of discrete distributions with countably infinite support is a challenging problem.
In this paper, we investigate the bias of the plug-in estimator $\hat{H}_n$ for the Shannon entropy $H(\bm{p})$ under an occupancy model whose frequency sequence exhibits regular variation with tail index $\alpha\in(0,1)$.
Using Poissonization and the theory of regular variation, we establish the asymptotic relation $|\E[\hat H_n] - H(\bm{p})| \sim n^{\alpha-1}L(n)C_\alpha$, where $L$ is a slowly varying function and $C_\alpha$ is an explicit constant depending only on $\alpha$ that admits an integral representation.
Our result shows that the asymptotic behavior of the bias of the plug-in estimator under power-law frequency distributions is determined by the tail behavior of the underlying distribution.
\end{abstract}


\section{Introduction}
Estimating the Shannon entropy of discrete probability distributions is an important problem in information theory, statistics, and related fields.
In entropy estimation, the plug-in estimator, which replaces the underlying distribution with the empirical distribution, is widely used as a standard approach; see, for example,~\cite{Basharin_1959,Hausser_2009}. 
Compared with the finite support setting, estimating the entropy of distributions with countably infinite support is more challenging and has therefore attracted considerable attention, particularly regarding consistency and convergence rates. 
In particular, Antos and Kontoyiannis~\cite{Antos_2001} analyzed the convergence rates and showed that, without imposing any restrictions on the class of distributions, no estimator can achieve a uniform rate of convergence. 
The regularly varying occupancy model, studied by Karlin~\cite{Karlin_1967} and Gnedin et al.~\cite{Gnedin_2007}, is known as a natural framework for describing heavy-tailed discrete distributions with infinite support. 
As a regular variation condition on the frequency sequence leads to a precise asymptotic theory, we study entropy estimation under the regularly varying occupancy model.

We investigate the bias of the plug-in estimator for the Shannon entropy under the regularly varying occupancy model.
Combining the Poissonization technique with the theory of regular variation, we derive the leading asymptotic term of the bias, expressed in terms of the tail index and the slowly varying function.
Moreover, whereas Antos and Kontoyiannis~\cite{Antos_2001} established convergence rates under polynomial tail conditions, our analysis extends their framework to general regularly varying tails and identifies the exact leading asymptotic constant.
Our result provides a precise asymptotic characterization of the bias under the regularly varying occupancy model.

The remainder of this paper is organized as follows.
Section~\ref{sec:setting} introduces the regularly varying occupancy model, the plug-in Shannon entropy estimator, and assumptions.
Section~\ref{sec:pois} presents the Poissonization technique and a de-Poissonization bound for the difference in expectations between the two plug-in estimators.
Section~\ref{sec:main} states the main results on the asymptotic bias of the plug-in Shannon entropy estimator and presents numerical simulations illustrating the asymptotic approximation.
The proofs are provided in Section~\ref{sec:proof}.

\section{Setting}\label{sec:setting}
Consider a discrete probability distribution $\bm{p}=(p_1,p_2,\ldots)$ with countably infinite support.
We assume that $p_j>0$ for all $j\ge1$ and $\sum_{j\ge1}p_j=1$. 
Let 
\[
\nu(\dx)=\sum_{j\ge1}\delta_{p_j}(\dx)
\]
be the counting measure associated with $\bm p$, and 
\[
\vec{\nu}(x)\coloneqq \nu([x,1]) \quad (x \in (0,1])
\]
the corresponding tail function.
The Shannon entropy $H(\bm{p})$ of $\bm{p}$ is given by
\[
H(\bm{p}) \coloneqq - \sum_{j \geq 1} p_j \log p_j = \int_0^1 (-x \log x) \nu(\dx).
\]
The following assumption will be used throughout the paper.

\begin{ass}\label{ass:1}
There exist $\alpha \in (0,1)$ and a slowly varying function $L$ such that
\begin{equation}
\vec{\nu}(x)=x^{-\alpha} L(x^{-1}). \label{def:assumption}
\end{equation}    
\end{ass}

\begin{rem}
As $\bm{p}$ is infinite-dimensional, the convergence of the infinite series defining $H(\bm{p})$ is not immediate; see, for example, \cite{Baccetti_2013}.
From Proposition 23 of \cite{Gnedin_2007} and Theorem 5.8 of \cite{Hashino_2026}, $H(\bm{p})$ is finite under Assumption~\ref{ass:1}.
\end{rem}

Recall that a slowly varying function $L$ satisfies
\[
\lim_{x\to\infty} \frac{L(\lambda x)}{L(x)}=1
\]
for any $\lambda>0$.
The condition \eqref{def:assumption} is standard in the literature on occupancy problems; see, for example,~\cite{Gnedin_2007,Karlin_1967}.
We refer to this framework as the regularly varying occupancy model.

Let $\bm{X}_n=(X_{n,j})_{j\ge1}$ be an infinite-dimensional random vector following the multinomial distribution with sample size $n (\geq 2)$ and probability vector $\bm{p}$, whose joint probability mass function is given by
\[
\Prob(X_{n,j} = n_j, \, j \in \mathbb{N}) = \frac{n!}{\prod_{j \geq 1} n_j!} \prod_{j \geq 1} p_j^{n_j}
\]
for non-negative integers $(n_j)_{j \geq 1}$ satisfying $\sum_{j \geq 1} n_j = n$.
The plug-in estimator $\hat{H}_n$ for $H(\bm{p})$ is given by
\begin{align*}
    \hat{H}_n 
    &\coloneqq  -\sum_{j\ge1} \frac{X_{n,j}}{n} \log\left(\frac{X_{n,j}}{n}\right) .
\end{align*}
In what follows, we investigate the asymptotic behavior of 
\[
|\bias(n)| \coloneqq |\E[\hat{H}_n] - H(\bm{p})|
= H(\bm{p}) - \E[\hat{H}_n] ,
\]
which is the absolute value of the bias of $\hat{H}_n$.

\section{Poissonization}\label{sec:pois}

We employ the Poissonization technique. 
Instead of a fixed sample size $n$, we consider a random sample size $N(t)$ following a Poisson distribution with mean $t > 0$, i.e., $N(t) \sim \Pois(t)$. 
Denote by $\bm{X}(t) = (X_j(t))_{j\ge1}$ the counts of observations in each category from the sample of size $N(t)$.
Then,
\[X_j(t)\sim\Pois(tp_j), \qquad j\in\mathbb N, \]
independently.
The Poissonized plug-in estimator $\hat{H}(t)$ is given by
\begin{align*}
    \hat{H}(t)
    &\coloneqq  -\sum_{j\ge1} \frac{X_{j}(t)}{t} \log\left(\frac{X_{j}(t)}{t}\right).
\end{align*}
Quantities defined in the Poissonized model will be evaluated at $t=n$.
The following lemma quantifies the difference between the expectations of the original and Poissonized estimators.

\begin{lem}\label{lem:de-pois}
If $H(\bm{p}) < \infty$, then
\[
 | \E[\hat H_n]-\E[\hat H(n)] | \leq n^{-1}.
\]
\end{lem}

\section{Main result}\label{sec:main}
The following theorem establishes the asymptotic behavior of the Poissonized absolute bias
\[
|\biasPois(t) |\coloneqq H(\bm p)-\E[\hat H(t)].
\]

\begin{thm}\label{thm:Main}
Let Assumption~\ref{ass:1} hold.
Then,
\[
|\biasPois(t)|
\sim t^{\alpha-1}L(t)C_\alpha
\]
as $t\to\infty$,
where
\begin{align*}
C_\alpha &\coloneqq \int_0^\infty y^{-\alpha}\psi(y)\,\dy,\\
\psi(y) &\coloneqq -\log y-1+\E\!\left[(P_y+1-y)\log(P_y+1)\right] \quad (y>0),
\end{align*}
with $P_y$ being a Poisson random variable with mean $y \,(>0)$.
\end{thm}
Combining Lemma~\ref{lem:de-pois} with Theorem~\ref{thm:Main} immediately yields the following corollary for the original fixed-sample model.

\begin{cor}\label{cor:main}
Let Assumption~\ref{ass:1} hold.
Then, as $n\to \infty$,
\[
|\bias(n)|
\sim n^{\alpha-1}L(n)C_\alpha.
\]
\end{cor}

Figure~\ref{fig:C_alpha} compares the theoretical constant $C_\alpha$ with the empirical coefficient 
\[
\hat{C}_\alpha(n) \coloneqq \frac{H(\bm{p}) - B^{-1} \sum_{i=1}^B \hat{H}^{(i)}_n}{n^{\alpha-1}L(n)} ,
\]
which is computed from $B=50000$ Monte Carlo replications with sample size $n=10^5$, where $\bm{p}$ is given by the Zeta distribution with parameter $\alpha^{-1}$.
The numerical results indicate that $C_\alpha$ monotonically increases with $\alpha$, exhibiting a sharp growth as $\alpha$ approaches $1$. 
Furthermore, the values of $\hat{C}_\alpha(n)$ are in close agreement with those of $C_\alpha$, consistent with the asymptotic estimates of Theorem~\ref{thm:Main}.

\begin{figure}[t]
\centering
\begin{subfigure}[b]{0.58\textwidth}
    \vspace{0pt}
    \centering
    \includegraphics[width=\textwidth]{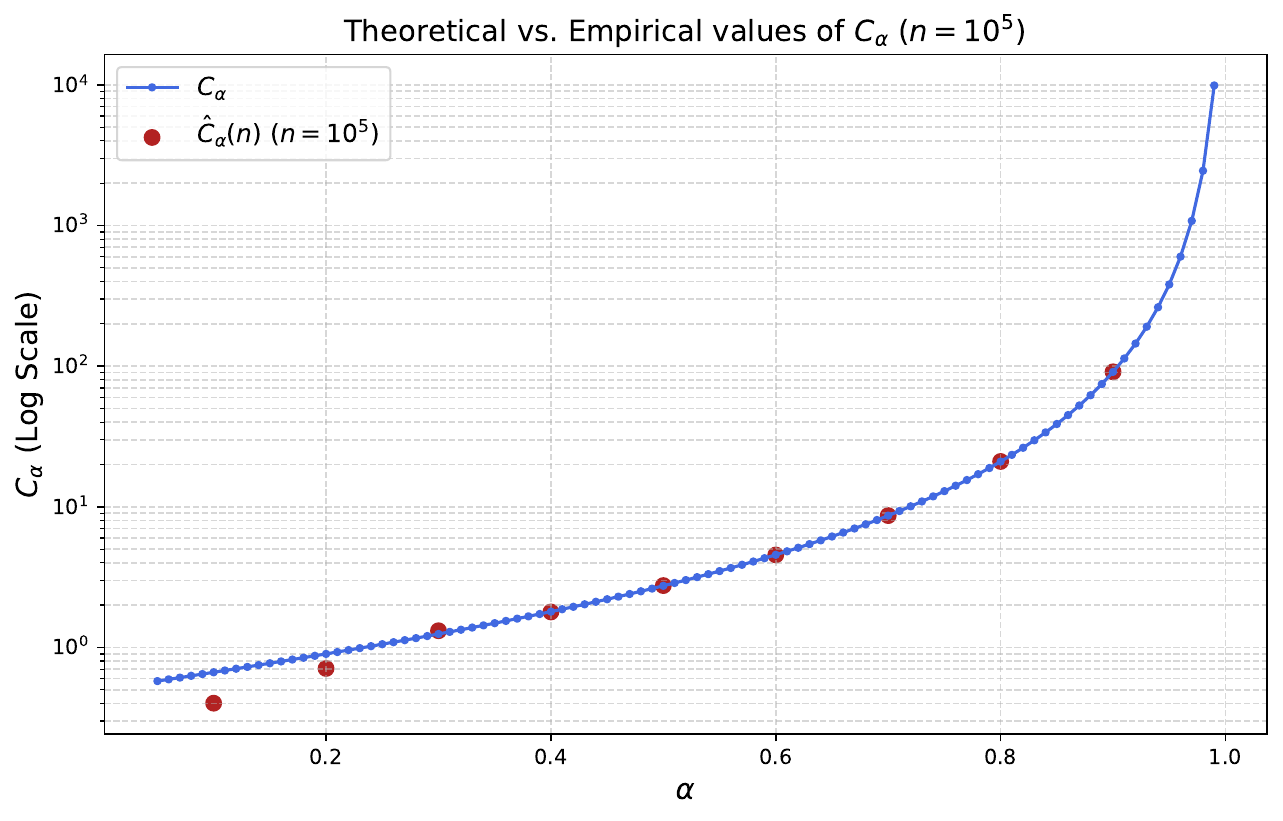}
    \caption{Logarithmic scale plot of $C_\alpha$.}
    \label{fig:C_alpha_plot}
\end{subfigure}
\hfill
\begin{subtable}[b]{0.4\textwidth}
    \vspace{0pt}
    \centering
    \small
    \begin{tabular}{ccc}
    \hline
    $\alpha$ & $C_\alpha$ & $\hat{C}_\alpha(n)$\\
    \hline
    0.1 & 0.666 & 0.402\\
    0.2 & 0.900 & 0.709\\
    0.3 & 1.247 & 1.314\\
    0.4 & 1.798 & 1.788\\
    0.5 & 2.744 & 2.754\\
    0.6 & 4.558 & 4.546\\
    0.7 & 8.659 & 8.660\\
    0.8 & 20.966 & 21.008\\
    0.9 & 91.065 & 91.132\\
    \hline
    \\
    \end{tabular}
    \caption{Theoretical and empirical values of $C_\alpha$.}
    \label{tab:C_alpha}
\end{subtable}

\caption{Comparison of theoretical and empirical values of $C_\alpha$.}
\label{fig:C_alpha}
\end{figure}

\section{Proofs}\label{sec:proof}

\subsection{Proof of Lemma~\ref{lem:de-pois}}
\begin{proof}[Proof of Lemma~\ref{lem:de-pois}]
Let $x\in(0,1)$.
Consider independent random variables $B_{n-1,x} \sim \Bin(n-1,x)$ and $B_{1,x} \sim \Bin(1,x)$, and define $B_{n,x} \coloneqq B_{n-1,x} + B_{1,x}$.
Moreover, consider a random variable $P_{nx} \sim \Pois(nx)$.
Since $H(\bm{p})$ is finite, $\E[\hat{H}_n]$ and $\E[\hat{H}(n)]$ are finite.
From the convexity of $x\log x$ and Lemma 2.1 of~\cite{Barman_2022},
\begin{align*}
    \left|\E[\hat H_n]-\E[\hat H(n)]\right|
    &=\left|\int_0^1\E\left[\frac{B_{n,x}}{n}\log\left(\frac{B_{n,x}}{n}\right)\right]-\E\left[\frac{P_{nx}}{n}\log\left(\frac{P_{nx}}{n}\right)\right]\nu(\dx)\right|\\
    &\le\frac{1}{n}\int_0^1\left|\E\left[B_{n,x}\log B_{n,x} \right]-\E\left[P_{nx}\log P_{nx} \right]\right|\nu(\dx)\\
    &=\frac{1}{n}\int_0^1\left(\E\left[P_{nx}\log P_{nx} \right]-\E\left[B_{n,x}\log B_{n,x} \right]\right)\nu(\dx).
\end{align*}
Moreover,
\[ 
\E[B_{n,x}\log B_{n,x}]=nx\E[\log(B_{n-1,x}+1)] , \quad 
\E[P_{nx}\log P_{nx}]=nx\E[\log(P_{nx}+1)] ,
\]
and
\begin{align*}
\E\left[\log(B_{n,x} + 1 )\right]
&= \E\left[\log(B_{n-1,x} +B_{1,x} + 1 )\right]\\
&= \E\left[\E\left[\log(B_{n-1,x} +B_{1,x} + 1 )\mid B_{n-1,x}\right]\right]\\
&=x\E\left[\log(B_{n-1,x} + 2 )\right]+(1-x)\E\left[\log(B_{n-1,x} + 1 )\right].
\end{align*}
Hence, from the concavity of $\log(x+1)$ and Lemma 2.1 of~\cite{Barman_2022},
\begin{align*}
    \left|\E[\hat H_n]-\E[\hat H(n)]\right|
    &\le\int_0^1x\left(\E\left[\log\left(P_{nx}+1\right)\right]-\E\left[\log\left(B_{n-1,x}+1\right)\right]\right)\nu(\dx)\\
    &\le\int_0^1x\left(\E\left[\log\left(B_{n,x}+1\right)\right]-\E\left[\log\left(B_{n-1,x}+1\right)\right]\right)\nu(\dx)\\
    &=\int_0^1x^2\left(\E\left[\log\left(B_{n-1,x}+2\right)\right]-\E\left[\log\left(B_{n-1,x}+1\right)\right]\right)\nu(\dx)\\
    &\le\int_0^1x^2\E\left[\frac{1}{B_{n-1,x}+1}\right]\nu(\dx)\\
    &=\int_0^1x^2\cdot\frac{1}{nx}(1-(1-x)^n)\nu(\dx)\\
    &\le\frac{1}{n}\int_0^1x\nu(\dx)=\frac{1}{n} .
\end{align*}
This completes the proof.
\end{proof}

\subsection{Proof of Theorem~\ref{thm:Main}}
We first establish the following lemma concerning $\psi(y)= -\log y-1+\E[(P_y+1-y)\log(P_y+1)]$ $(y > 0)$, where $P_y \sim \Pois(y)$.

\begin{lem}\label{lem:psi}
For all $y>0$,
\begin{align*}
\begin{array}{ll}
|\psi(y)| \leq C(1-\log y) &(0 < y \leq 1),\\
|\psi(y)| \leq y^{-1} &(y \geq 1 ).
\end{array}
\end{align*}
\end{lem}

\begin{proof}[Proof of Lemma~\ref{lem:psi}]
    We first consider the case $0<y\le1$.
    It follows from the Jensen inequality that
    \begin{align*}
        \psi(y)
        &=-\log y-1+\E[(P_y+1-y)\log(P_y+1)]\\
        &\ge-\log y-1+\log(y+1)\\
        &=\log(1+y^{-1})-1.
    \end{align*}
    On the other hand, since 
    \[ 
    \E[P_y\log P_y] =  y\E[\log(P_y+1)],
    \]
    it also follows from the Jensen inequality that
    \begin{align*}
        \psi(y)
        &=-\log y-1+\E[(P_y+1)\log(P_y+1)-P_y\log P_y]\\
        &\le-\log y-1+(y+1)\log(y+1)-y\log y\\
        &=(y+1)\log (1 + y^{-1})-1.
    \end{align*}
    Therefore,
    \[
    \log (1 + y^{-1}) - 1 \leq \psi(y)\le(y+1) \log(1 + y^{-1})-1.
    \]
    Hence,
    \[
    |\psi(y)| \leq \max\{|\log(1 + y^{-1}) - 1 |, |(y+1) \log(1+y^{-1})-1 | \}.
    \]
    Since $\log(1 + y^{-1})\le\log 2-\log y$,
    \begin{align*}
        |\log(1 + y^{-1})-1 |
        &\le(1+\log2)-\log y\\
        |(y+1) \log(1 + y^{-1}) - 1 |
        &\le(2\log2+1)-2\log y.
    \end{align*}
    This proves the desired result for $0<y\le1$.
    
    Next, we consider the case $y\ge1$.
    Observe that
    \begin{align*}
        \psi(y)
        &=-\log y-1+\E[P_y\log (P_y+1)]-y\E[\log(P_y+1)]+\E[\log(P_y+1)]\\
        &=y\E\left[\log\left(1+\frac{1}{P_y+1}\right)\right]+\E\left[\log\left(\frac{P_y+1}{y}\right)\right]-1.
    \end{align*}
    Denote the first term by $T_1$ and the second term by $T_2$.
    Since $x/(x+1) \le\log(1+x)\le x$,
    \[
    \frac{1}{P_y+2}\le\log\left(1+\frac{1}{P_y+1}\right)\le \frac{1}{P_y+1}.
    \]
    Therefore,
    \[
    y\E\left[\frac{1}{P_y+2}\right]=1-\frac{1-e^{-y}}{y}\le T_1\le y\E\left[\frac{1}{P_y+1}\right]=1-e^{-y}
    \]
    Hence, $1 - y^{-1} \leq T_1 \leq 1$.
    By the Jensen inequality, $T_2 \leq \log(1 + y^{-1}) \leq y^{-1}$, while 
    \[
    \E\left[\log\left(\frac{P_y+1}{y}\right)\right]\ge \E\left[1-\frac{y}{P_y+1}\right]=e^{-y}.
    \]
    Therefore, $0 \leq e^{-y} \leq T_2 \leq y^{-1}$.
    Combining the above estimates yields $- y^{-1} \leq \psi(y) \leq y^{-1}$.
    This proves the desired result for $y\ge1$, and hence completes the proof.
\end{proof}

We recall the following lemma.

\begin{lem}[Potter bounds; \cite{Bingham_1987}, Theorem~1.5.6(ii)]\label{lem:potter}
    If $\ell: [0,\infty) \to (0,\infty)$ is a slowly varying function at infinity that is bounded away from $0$ and $\infty$ on every compact subset of $[0,\infty)$, then for every $\delta>0$ there exists $A=A(\delta)>1$ such that 
    \[
    \frac{\ell(y)}{\ell(x)}\le A\max\left\{\left(\frac{y}{x}\right)^\delta,\left(\frac{y}{x}\right)^{-\delta}\right\}\quad(x>0,y>0).
    \]
\end{lem}

We then provide the proof of Theorem~\ref{thm:Main}.

\begin{proof}[Proof of Theorem~\ref{thm:Main}]
By the definition of $\hat H(t)$,
\[
\E[\hat H(t)]=-\int_0^1x\E\left[\log\left(\frac{P_{tx}+1}{t}\right)\right] \nu(\dx),
\]
where $P_{tx}\sim\Pois(tx)$.
This implies that
\[
|\biasPois(t)| 
=\int_0^1x\Psi(tx)\nu(\dx)
\]
with $\Psi(tx)=\E[\log(P_{tx}+1)-\log (tx)]$.
In particular,
\[
x\Psi(tx)
= -x\log(tx) + \sum_{k \geq0}  u_k(x),
\]
where 
\[
u_k(x)\coloneqq  \frac{t^k x^{k+1} e^{-tx}}{k!} \log(k+1) \quad (0<x<1).
\]

Differentiating yields
\[
u_k'(x)=\frac{(tx)^ke^{-tx}}{k!}(k+1-tx)\log(k+1) .
\]
Hence,
\[
|u'_k(x)|\le\frac{t^k}{k!}(k+1+t)\log(k+1) \eqqcolon M_k .
\]
Moreover,
\[
\lim_{k\to\infty}\frac{M_{k+1}}{M_k}=\lim_{k\to\infty}\frac{t}{k+1}\times\frac{k+2+t}{k+1+t}\times\frac{\log(k+2)}{\log(k+1)}=0.
\]
Therefore, by the ratio test, $\sum_{k \geq 0} M_k$ converges.
By the Weierstrass $M$-test, $\sum_{k \geq 0} u_k'(x)$ converges uniformly on $(0,1)$.
Therefore,
\begin{align*}
    \frac{ \mathrm{d} }{\dx}(x\Psi(tx))
    &=-\log(tx)-1+\sum_{k \geq 0} \frac{(tx)^ke^{-tx}}{k!}(k+1-tx)\log(k+1)\\
    &=-\log(tx)-1+\E[(P_{tx}+1-tx)\log(P_{tx}+1)]\\
    &=\psi(tx).
\end{align*}
Furthermore, integration-by-parts yields
\begin{align*}
    \int_0^1x\Psi(tx)\nu(\dx)
    &=-\Psi(t)\vec\nu(1)+\lim_{\epsilon\to0}\epsilon\Psi(t\epsilon)\vec\nu(\epsilon)+\int_0^1\vec\nu(x) \mathrm{d} (x\Psi(tx))\\
    &=\int_0^1\vec\nu(x)\psi(tx) \dx.
\end{align*}
It follows from $0 \le \E[\log(P_{t\epsilon}+1)] \leq t\epsilon$ that $|\Psi(t\epsilon)| = O(-\log\epsilon)$ as $\epsilon\to0$, which, together with Assumption~\ref{ass:1}, implies $\lim_{\epsilon\to0}\epsilon\Psi(t\epsilon)\vec\nu(\epsilon) = 0$.
The change of variables $y=tx$ yields
\[
| \biasPois(t) |
=\int_0^t\psi(y)\vec\nu(y/t)t^{-1} \dy.
\]
Since $\vec\nu(x)= x^{-\alpha}L\left(1/x\right)$,
\begin{align*}
| \biasPois(t) |  
&=\int_0^t\left(\frac{y}{t}\right)^{-\alpha}L\left(\frac{t}{y}\right)\psi(y)t^{-1} \dy\\
&=t^{\alpha-1}L(t)\int_0^ty^{-\alpha}\frac{L(t/y)}{L(t)}\psi(y) \dy.
\end{align*}
Hence,
\[
\frac{ | \biasPois(t) |  }{t^{\alpha-1}L(t)}=\int_0^\infty y^{-\alpha}\frac{L(t/y)}{L(t)}\psi(y)I\{y\le t\} \dy.
\]
Therefore, it suffices to determine 
$\lim_{t\to\infty}\int_0^\infty f_t(y)\dy$, where 
\[ f_t(y)\coloneqq y^{-\alpha} \frac{L(t/y)}{L(t)}\psi(y)I\{y\le t\} \quad (y > 0). \]
By Lemma~\ref{lem:potter}, for any $\delta \in (0, \min\{\alpha,1-\alpha\})$,
\[
|f_t(y)|\le Ay^{-\alpha}\max\{y^{\delta},y^{-\delta}\}|\psi(y)|.
\]
Define 
\[ g(y)=Ay^{-\alpha}\max\{y^{\delta},y^{-\delta}\}|\psi(y)| \quad (y>0),\]
which is integrable.
Indeed,
\[
\int_0^\infty g(y) \dy
=A\left(\int_0^1y^{-(\alpha+\delta)}|\psi(y)| \dy 
+ \int_1^\infty y^{-(\alpha-\delta)}|\psi(y)| \dy \right) ,
\]
and
\begin{align*}
\int_0^1y^{-(\alpha+\delta)}|\psi(y)| \dy
&\leq C\int_0^1y^{-(\alpha+\delta)}(1-\log y) \dy <\infty , \\
\int_1^\infty y^{-(\alpha-\delta)}|\psi(y)| \dy
&\leq \int_1^\infty y^{- (\alpha - \delta + 1)} \dy
<\infty
\end{align*}
by Lemma~\ref{lem:psi}.
Therefore, by the dominated convergence theorem,
\begin{align*}
    \lim_{t\to\infty}\int_0^\infty y^{-\alpha}\frac{L(t/y)}{L(t)}\psi(y)I\{y\le t\} \dy
    &=\int_0^\infty\lim_{t\to\infty} y^{-\alpha}\frac{L(t/y)}{L(t)}\psi(y)I\{y\le t\} \dy \\
    &=\int_0^\infty y^{-\alpha}\psi(y) \dy
    = C_\alpha.
\end{align*}
This completes the proof.
\end{proof}

\section*{Acknowledgments}
The second author was supported in part by Japan Society for the Promotion of Science KAKENHI Grant Number 25K07133.


\end{document}